\documentclass[10pt,bezier]{article}
\usepackage{amsmath,amssymb,amsfonts,graphicx,caption,subcaption}
\usepackage[colorlinks,linkcolor=red,citecolor=blue]{hyperref}

\newtheorem{prethm}{{\bf Theorem}}[section]
\newenvironment{thm}{\begin{prethm}{\hspace{-0.5
em}{\bf.}}}{\end{prethm}}
\newtheorem{prepro}{{\bf Theorem}}

\newtheorem{precor}[prethm]{{\bf Corollary}}

\newtheorem{preconj}[prethm]{{\bf Conjecture}}
\newenvironment{conj}{\begin{preconj}{\hspace{-0.5
em}{\bf.}}}{\end{preconj}}
\newtheorem{preremark}[prethm]{{\bf Remark}}

\newtheorem{prelem}[prethm]{{\bf Lemma}}

\newtheorem{preque}[prethm]{{\bf Problem}}

\newtheorem{prealphthm}{{\bf Problem}}

\newenvironment{alphprob}
{
\begin{prealphthm}{\hspace{-0.5 em}
{\bf\ }}}{
\end{prealphthm}
}

\newtheorem{preobserv}[prethm]{{\bf Observation}}

\newtheorem{predef}[prethm]{{\bf Definition}}

\newtheorem{preproposition}[prethm]{{\bf Proposition}}

\newtheorem{preproof}{{\bf Proof.}}
\newtheorem{preprooff}{{\bf Proof}}

\newenvironment{proof}[1]{\begin{preproof}{\rm
#1}\hfill{$\Box$}}{\end{preproof}}

\newtheorem{preproofF}{{\bf Proof of}}

\title{\bf\Large 
The existence of uniform hypergraphs for which interpolation property of complete coloring fails
}

\author{Nastaran Haghparast\thanks{Department of Mathematics and Computer Sciences, Amirkabir University of Technology, Tehran, Iran. E-mail: {\tt nhaghparast@aut.ac.ir}}, Morteza Hasanvand\thanks{
Department of Mathematical Sciences, Sharif University of Technology, Tehran, Iran.
E-mail: {\tt hasanvand@alum.sharif.edu}}, and 
Yumiko Ohno\thanks{Research Initiatives and Promotion Organization, Yokohama National University,
79-7, Tokiwadai, Hodogaya-ku, Yokohama 240-8501, Japan. E-mail: {\tt ohno-yumiko-hp@ynu.ac.jp}}
}

\date{}

\begin{document}
\maketitle
\begin{abstract}{
In 1967 Harary, Hedetniemi, and Prins showed that every graph $G$ admits a complete $t$-coloring for every $t$ with $\chi(G) \le t \le \psi(G)$, where $\chi(G)$ denotes the chromatic number of $G$ and $\psi(G)$ denotes the achromatic number of $G$ which is the maximum number $r$ for which $G$ admits a complete $r$-coloring. Recently, Edwards and Rz\c a\.{z}ewski (2020) showed that this result fails for hypergraphs by proving that for every integer $k$ with $k\ge 9$, there exists a $k$-uniform hypergraph $H$ with a complete $\chi(H)$-coloring and a complete $\psi(H)$-coloring, but no complete $t$-coloring for some $t$ with $\chi(H)< t<\psi(H)$. They also asked whether there would exist such an example for $3$-uniform hypergraphs and posed another problem to strengthen their result. In this paper, we generalize their result to all cases $k$ with $k\ge 3$ and settle their problems by giving several kinds of $3$-uniform hypergraphs. In particular, we disprove a recent conjecture due to Matsumoto and the third author (2020) who suggested a special family of $3$-uniform hypergraph to satisfy the desired interpolation property.
} {\small
}
\end{abstract}
%
%
%
%
%
%
%
%
%
%
\section{Introduction}
In this paper, all hypergraphs are considered simple. Let $H$ be a hypergraph. 
The vertex set and the hyperedge set of $H$ are denoted by $V(H)$ and $E(H)$, respectively. 
A vertex subset of $V(H)$  is said to be {\it independent}, if there is no hyperedge of $H$ including two different vertices of it.
The {\it incidence graph} of $H$ refers to a bipartite graph $G$ with $V(G)=V(H)\cup E(H)$ in which a vertex $v\in V(H)$ and a hyperedge $e\in E(H)$ are adjacent in $G$ if  and only if $v\in e$. 
A hypergraph is said to be {\it $k$-uniform}, if the size of all of  hyperedges are the same number $k$.
We say that a vertex set $S$ {\it covers} a hyperedge $e$,  if $S$ includes at least one vertex of $e$.
A {\it face hypergraph} refers a hypergraph obtained from a embedded  graph $G$ whose vertices are the same vertices of $G$
and there is a one-to-one correspondence between the faces of $G$ and hyperedges of $H$ such that each hyperedge of $H$ consists all vertices of  its corresponding face.
This concept was introduced by K\"{u}ndgen and Ramamurthi~\cite{Kundgen-Ramamurthi}.
The minimum number of  colors needed to color the vertices of $H$ such that any two  vertices lying in the same hyperedge have different colors (proper property) is denoted by $\chi(H)$.
A {\it complete $t$-coloring} of a $k$-uniform hypergraph $H$ is a coloring of whose vertices, using $t$ colors, 
such that any two vertices  lying in the same hyperedge  have different colors, and also every arbitrary set of $k$
 different colors appears in at least one hyperedge.
Note that an arbitrary uniform hypergraph may have not a complete coloring, see~\cite{DEBSKI201760}.
If $H$ has a complete $t$-coloring,  we denote by $\psi(H)$ the maximum number of such integers $t$; 
otherwise, we define $\psi(H)=0$.   
  The numbers $\chi(H)$ and $\psi(H)$ are called the {\it chromatic number}  and the {\it achromatic number} of $H$, respectively. 
It was proved in~\cite{EDWARDS2020111673, Matsumoto-Ohno} 
that a given uniform hypergraph $H$
  may  have not a complete $\chi(H)$-coloring even if it admits a complete coloring.
We say that a hypergraph $H$ satisfies {\it interpolation property}, if it admits a complete $t$-coloring for every integer $t$ with $\chi(H) \le s < t \le \psi(H)$, provided that $H$ has a complete $s$-coloring. 

In 1967 Harary,  Hedetniemi, and Prins studied interpolation property for complete coloring of graphs and established the following result.
\begin{thm}{\rm (\cite{Harary-Hedetniemi-Prins})}\label{thm:graphs}
{Every graph $G$ admits a complete  $t$-coloring for every $t$ with  $\chi(G)\le  t\le \psi(G)$.
}\end{thm}

Recently, Edwards and Rz\c a\.{z}ewski (2020) showed that Theorem~\ref{thm:graphs} cannot be developed to $k$-uniform hypergraphs for all integers $k$ with $k\ge 9$. 
\begin{thm}{\rm (\cite{EDWARDS2020111673})}\label{thm:k:9}
{Let $k$ be a positive integer with $k\ge 9$. 
There exists a $k$-uniform hypergraph $H$  which has a complete  $\chi(H)$-coloring, 
 and a complete $\psi(H)$-coloring, but no complete coloring for some $t$ 
with  $\chi(H)< t<\psi(H)$.
}\end{thm}

In addition, they asked the following two problems for generalizing  Theorem~\ref{thm:k:9} to $3$-uniform hypergraphs,
and  for studying a weaker version of interpolation property of complete coloring of hypergraphs.
\begin{alphprob}{\rm (Edwards and Rz\c a\.{z}ewski (2020) \cite{EDWARDS2020111673})}\label{prob:3}
{Does there exist a $3$-uniform example of a hypergraph for which interpolation fails?
}\end{alphprob}
\begin{alphprob}{\rm (Edwards and Rz\c a\.{z}ewski (2020) \cite{EDWARDS2020111673})}\label{prob:4}
{Does there exist  a  hypergraph $H$ with a complete $\chi(H)$-coloring and a complete $\psi(H)$-coloring, 
but no  complete $t$-coloring for every $t$ satisfying $\chi(H)< t<\psi(H)$ in which  $\psi(H)\ge \chi(H)+2$?
}\end{alphprob}

In this paper, we generalize Theorem~\ref{thm:k:9} to  all cases $k$ with $k\ge 3$ by modifying some parts of their proof.
In Section~\ref{sec:3-uniform}, we  answer Problem~\ref{prob:4} positively by giving several kinds of  $3$-uniform hypergraphs, which consequently shows that  the answer of Problem~\ref{prob:3} is positive. In particular, we form the following stronger assertion.
\begin{thm}
{There exists a  $3$-uniform hypergraph $H$ with a complete $\chi(H)$-coloring and a complete $\psi(H)$-coloring, 
but no  complete $t$-coloring for every $t$ satisfying $\chi(H)< t<\psi(H)$ in which  $\psi(H)\ge 2\chi(H)$.
}\end{thm}

Recently, Matsumoto and the third author (2020) investigated complete coloring for a special family of face hypergraphs
 using terms of facial complete coloring of planar triangulations. They put forward the following conjecture in their paper 
 to suggest a family of hypergraphs satisfying interpolation property.  In the rest of this paper, we disprove this conjecture by a particular hypergraph of order $12$ which seems to be the unique exceptional example for this conjecture.  
It is known that a planar triangulation  is $3$-colorable if and only if  whose  degrees are even~\cite{Tsai-West}. 
\begin{conj}{\rm (Matsumoto and Ohno (2020) \cite{Matsumoto-Ohno})}\label{conj}
{Let $H$ be a $3$-uniform face hypergraph obtained from a planar triangulation. 
If $H$ is   $3$-colorable, then it admits a complete $t$-coloring for every $t$ with $\chi(H) \le t\le \psi(H)$.
}\end{conj}
%
\section{The existence of  uniform hypergraphs for which interpolation property fails}
\label{sec:k-uniform}
The following theorem makes a stronger version for Theorem~\ref{thm:k:9}. 
\begin{thm}
{Let $k$ be a positive integer with $k\ge 3$. There exists a $k$-uniform hypergraph $H$  
which has a complete  $\chi(H)$-coloring and a complete $\psi(H)$-coloring, but no complete $t$-coloring for some $t$ with  $\chi(H)< t<\psi(H)$.
}\end{thm}
\begin{proof}{ 
We may assume that $k\ge 4$, as the assertion holds for $k=3$ with respect to Theorem~\ref{thm:9v}. 
Let $r$ be a large enough integer number compared to  $k$. 
Define $ H$ to be the $k$-uniform hypergraph with 
$V (H) =\{v_{i,j}:1\le i\le k, \,1\le j\le r\}$
 and $E(H) ={E}_1\cup {E}_2$ such that
$${E}_1= \{\{v_{i,p_i} :  1\leq i\leq k\}: (p_1,\ldots ,p_k) \in \mathcal{A}  \text{\, and } f(p_1,\ldots , p_k) \leq1\},\text{ and }$$
$${E}_2=\{\{v_{i,p_i} :  1\leq i\leq k\}: (p_1,\ldots ,p_k) \in \mathcal{A}  \text{\, and } p_1<\cdots < p_k\},$$
where $\mathcal{A}$  denotes the set of all sequences  $(p_1,\ldots, p_k)$ such that all $p_i$ are distinct and $1 \leq p_i \leq r$ and $f(p_1,\ldots , p_k) =|\{ (i, j) : |p_i - p_j | = 1 \text{ and } 1 \leq i < j \leq k\}|$. 
We call the $i$-th  part of $H$  as the set of all vertices $v_{i,j}$ with $1\le j\le r$, and
 call the $j$-th position  of $H$ as the set of all vertices $v_{i,j}$ with $1\le i\le k$.
According to this construction, one can  prove the following three assertions:
\begin{enumerate}{
\item [(a1)] There is no hyperedge including two vertices of the same position.\label{Condition 1}
\item [(a2)] There is no hyperedge including two vertices of the same part.\label{Condition 2}
\item [(a3)] For any two vertices in different parts and different positions,  there is a hyperedge including them.\label{Condition 3}
}\end{enumerate}
%
We  prove  only the last assertion as the other ones are obvious. 
 Let $v_{i,j}$ and $v_{i',j'}$ be two arbitrary vertices of $H$ 
in different  parts and different positions  so that $i\neq i'$ and $j\neq j'$.
Since $r$ is large enough, there is an integer $s$ with $1\le s \le r$ such that
$ \{j,j'\}  \cap \{s,\ldots, s+2k+2\}=\emptyset$.
Consider the sequence $(p_1,\ldots, p_k)$ satisfying $p_i=j$, $p_{i'}=j'$, and $p_{t}=s+2t$ 
for every $t\in \{1,\ldots, k\}\setminus \{i,i'\}$. 
Obviously, this sequence  is in $\mathcal{A}$ and $f(p_1,\ldots, p_k)\le 1$. 
Thus the hyperedge  corresponding to this sequence must be in $E_1$. 
Note that this hyperedge includes both of $v_{i,j}$ and $v_{i',j'}$. Hence the claim holds.
%

 To show that this hypergraph has a complete $k$-coloring, we take color set 
$\{c_1, c_2,\ldots,c_k\}$ and for each $1\leq i \leq k$, 
we color  all vertices in the $i$-th part with the color  $c_i$. By $(a_2)$ this is a 
proper coloring and each hyperedge contains  all $k$ colors.
 For complete $r$-coloring,  we take a color  set  $\{c_1, c_2,\ldots,c_r\}$ 
and for each $1\leq j \leq r$,
 we color  all vertices in the $j$-th position  with the color  $c_j$. 
 According to $(a_1)$, it is a proper coloring. 
In addition, if  $\{c_{p_1}, c_{p_2},\ldots, c_{p_k}\}$ is a $k$-subset of
  $\{c_1, c_2,\ldots,c_r\}$ with $p_1< \cdots< p_k$, 
then the hyperedge $\{v_{1,p_1}, v_{2,p_2},\ldots, v_{k,p_k}\}$ of  ${E}_2$ 
contains  this color  set.
Therefore, $\chi(H) = k$ and $\psi(H)\ge r$. 

Now, we show that $H$ has no complete $t$-coloring for every integer $t$ with $\frac{k-2}{k-1}r+k+1 \le t<r$. 
Suppose,  to the contrary, that $H$ has a complete $t$-coloring using colors $c_1, \ldots,c_t$.
Define $X$ to be the set of colors  appearing in at least two parts and 
define $Y$ to be the set of colors  appearing in only one part. 
We are going to prove the following two assertions:
\begin{enumerate}{
\item [(b1)] Each color of $X$ appears in only one position and all vertices of this position colored only by this color.\label{Condition 1}
\item [(b2)]  Each part has only one color from $Y$ so that $|Y|=k$ and $|X| = t - k$.\label{Condition 2}
}\end{enumerate}
Consider a color  $x \in X$.   
If $x\in X$   occurred in more than one
position, then by the definition of $X$, 
there must be two vertices having the same color $x$ with different parts and different positions.
Thus  by  $(a_3)$ there  is a hyperedge including both of them. 
This shows that  the coloring is not proper, a contradiction. 
Thus all occurrences of $x$ are in the same position. 
Now, since $|X|< r$, there is one position whose colors are not in $X$.
 In other words, there are $k$ vertices with different parts whose colors are in $Y$.
On the other hand, each part contains at most one color of $Y$; otherwise,  if two colors of $Y$ are in the same part, then by (a2) there is no hyperedge including them which is impossible.
 Therefore,  $|Y| = k$ and $|X| = t-k$.
Consequently, we can define $y_i$ to be the unique color in $Y$ appearing in the $i$-th part, where $1\le i\le k$.
Assume that the color $x\in X$ appears in the $j$-th position.
We are going to show that all vertices of this position are colored by this color.
If we consider a given arbitrary vertex $v_{i,j}$ of this position, then there is one hyperedge of $H$ containing all colors of the set
$\{y_1, \ldots ,y_{i-1}, x, y_{i+1}, \ldots, y_k\}$.
 Let   $(p_1,\ldots, p_k)\in \mathcal{A}$ be the sequence corresponding to this hyperedge. 
Obviously, the color of $v_{t,p_t}$ must be $y_t$  for every $t\in \{1,\ldots, k\}$ with $t\neq i$.
Thus  the color $x$ must be appeared on the $i$-th part, and so the vertex $v_{i,j}$ must be colored with $x$.
Therefore,  all of vertices of  the $j$-th position are  colored with the color $x$. 
Hence the assertions hold.

Obviously, there are $r-|X|$ positions are not colored by colors of $X$. 
Since $r-|X| \le r/(k-1)-1$, we can conclude that there are $k-1$
 consecutive positions $\{s, s+1,\ldots, s+k-2\}$ of $H$  colored only with colors of $X$. 
Define $Z$ to be the set of all those $k-1$ colors along with the color $y_2$. 
By the assumption, there is  a hyperedge $e \in E(H) $ including all colors of $Z$. 
Let $(p_1,\ldots, p_k)\in \mathcal{A}$ be the sequence  corresponding to this hyperedge. 
Obviously,  by $(b2)$, the vertex $v_{2,p_2}$ must be colored by  $y_2$. 
We know that $\{p_1, \ldots, p_k\}\setminus \{p_2\}= \{s, s+1,\ldots, s+k-2\}$. 
Since $k\ge 4$, there must be three integers $a,b, c\in \{1,\ldots, k\}$ such that
$\{p_{a},p_{b}, p_{c}\}=\{s,s+1,s+2\}$.
Thus $f(p_1, \ldots, p_k)\ge 2$ and so $e\notin {E}_1$. 
Moreover, according to the situation of the position containing the color $y_2$, we have either 
 $\max\{p_1,p_3\}<p_2$  
or 
$p_2<\min\{p_1,p_3\}$ and so $e \notin {E}_2$.
This is a contradiction. Hence the theorem is proved.
}\end{proof}
%

\section{Answering to  Problem~\ref{prob:4} by $3$-uniform hypergraphs}
\label{sec:3-uniform}
In this section, we are going to answer to Problem~4 in~\cite{EDWARDS2020111673} by giving several kinds of $3$-uniform hypergraphs.
%
\subsection{A hypergraph of order $9$}
A positive answer to Problem~\ref{prob:4} is given in the following theorem. 
\begin{thm}\label{thm:9v}
{There exists a $3$-uniform hypergraph $H$ of order $9$ with a complete $\chi(H)$-coloring and a complete $\psi(H)$-coloring, 
but no  complete $t$-coloring for every $t$ satisfying $\chi(H)< t<\psi(H)$ in which  $\psi(H)\ge \chi(H)+2$.
}\end{thm}
\begin{proof}
{Let $H$ be the $3$-uniform hypergraph of order 9 whose incidence graph is shown in Figure~\ref{fig:comp3}.
If $H$ has a complete $k$-coloring for $k \ge 6$,
then it has at least twenty hyperedges.
However, $H$ has exactly ten hyperedges and hence $\psi(H) \le 5$.
In fact, $H$ has a complete $3$-coloring and a complete $5$-coloring (see Figures~\ref{fig:comp3} and~\ref{fig:comp5}, respectively). Therefore, $\chi(H) = 3$ and $\psi(H)=5$.

\begin{figure}[h]
 \begin{minipage}{0.5\hsize}
  \centering
   \includegraphics[scale = 1]{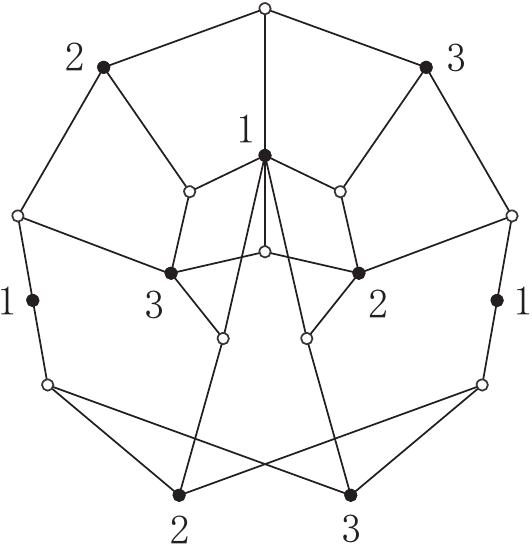}
  \caption{A complete $3$-coloring of $H$}
  \label{fig:comp3}
 \end{minipage}
 \begin{minipage}{0.5\hsize}
 \centering
  \includegraphics[scale = 1]{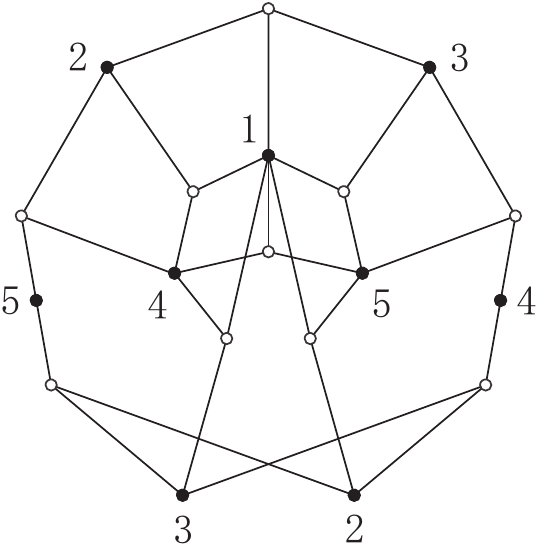}
  \caption{A complete $5$-coloring of $H$}
  \label{fig:comp5}
 \end{minipage}
\end{figure}

Next, we show that $H$ has no complete $4$-coloring.
Suppose, to the contrary, that $H$ has a complete $4$-coloring using colors $c_1,\ldots, c_4$.
Since $H$ has nine vertices, there exists at least one color appearing on at least three vertices of $H$, say color $c_1$.
Note that those vertices with the same color form an independent set.
It is easy to check 
that there are exactly three independent sets of $H$ with size three (which shown as vertices numbered by $1, 2$ and $3$  in Figure~\ref{fig:comp3}).
Since the vertices of every such vertex set cover all hyperedges of $H$,
the triad $\{c_2, c_3, c_4\}$ does not appear on any hyperedge of $H$.
Hence $H$ has no complete $4$-coloring and so it is a desired hypergraph.
}\end{proof}
%
%
\subsection{A $3$-regular $3$-uniform hypergraph of order $15$}

Another positive answer to Problem~\ref{prob:4} is given in the next theorem.
\begin{thm}
{There exists a $3$-uniform $3$-regular hypergraph of order $15$ with a complete $\chi(H)$-coloring and a complete $\psi(H)$-coloring, but no complete $t$-coloring for every $t$ satisfying $\chi(H)< t<\psi(H)$ in which  $\psi(H)\ge \chi(H)+2$.
}\end{thm}
\begin{proof}
{Let $H$ be the hypergraph with the vertex set
 $\{v_{i,j}:1\le i\le 3,  1\le j\le 5\}$ consisting of those hyperedges
 $e_{ij}$ with $1\le i\le 3$ and  $1\le j\le 5$
 in which 
 $$e_{ij}=\{v_{i,j+1}\} \cup \{v_{t,j}:1\le t\le 3, t\neq i\} ,$$ 
where  $v_{i,6}=v_{i, 1}$. 
The incidence graph of this hypergraph is shown in Figure~\ref{fig:15-comp3}.
Obviously, $H$ is $3$-uniform and $3$-regular.
If $H$ has a complete $k$-coloring for $k \ge 6$, then $H$ must have  at least twenty hyperedges.
However, $H$ has exactly fifteen hyperedges and hence $\psi(H) \le 5$.
In fact, $H$ has a complete $3$-coloring and a complete $5$-coloring (see Figures~\ref{fig:15-comp3} and~\ref{fig:15-comp5}, respectively).
Therefore, $\chi(H) = 3$ and $\psi(H)=5$.

\begin{figure}[h]
 \begin{minipage}{0.5\hsize}
  \centering
   \includegraphics[scale = 1]{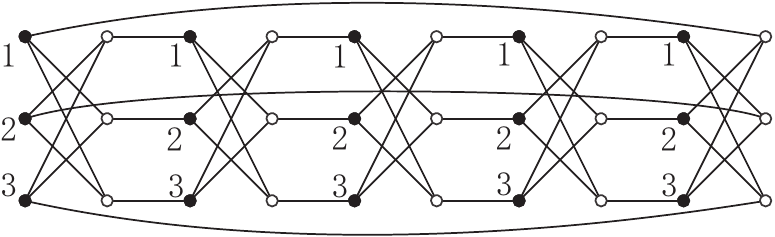}
  \caption{A complete $3$-coloring of $H$}
  \label{fig:15-comp3}
 \end{minipage}
 \begin{minipage}{0.5\hsize}
 \centering
  \includegraphics[scale = 1]{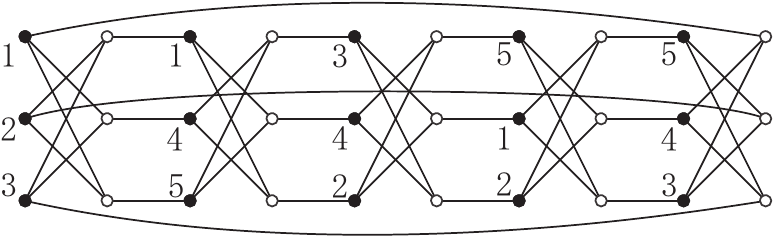}
  \caption{A complete $5$-coloring of $H$}
  \label{fig:15-comp5}
 \end{minipage}
\end{figure}

Suppose, to the contrary, that $H$ has a complete $4$-coloring using colors $c_1,\ldots, c_4$.
Since $|V(H)| = 15$, there exists a color appearing on at least four vertices, say $c_1$.
Define $X_i= \{v_{i,j}: 1\le j\le 5\}$ for each $i$ with $1\le i \le 3$.
According to the construction of $H$, 
it  is not difficult to  check that every independent set of size four must be a subset of $X_1$, $X_2$, or $X_3$.
Hence the color $c_1$ only appears on vertices of a set $X_t$ where $1\le t\le 3$.
If $c_1$ appears on five vertices, then  it must appear on all vertices of $X_t$.
In this case, the triad $\{c_2, c_3, c_4\}$ does not appear,
because  all hyperedges of $H$ are covered by the vertices of $X_t$.
Therefore, each color appears on at most four vertices.
Since $H$ has $15$ vertices,  every color must appear on four vertices, except one color which appears on three vertices.
We may assume that  for each $i\in \{1,2,3\}$, the 
color $c_i$  appears on exactly four  vertices of $X_i$.
Then the remaining three vertices are colored by  $c_4$ so that each $X_i$ includes exactly one of them.
Let us define $Y_j = \{v_{i,j}: 1\le i\le 3\}$ for each $j$ with $1\le j \le 5$.
It is easy to check that
if a vertex in $X_i$ and a vertex in $X_{i'}$ are colored by the same color provided that $i\neq i'$,
both of them cannot be in the set $Y_j\cup Y_{j+1}$ for all $j\in \{1,\ldots, 5\}$; where $Y_{6}=Y_1$. 
Now, since three vertices are colored by $c_4$ and each $X_i$ includes exactly one of them, we derive a contradiction.
Therefore, $H$ has no complete $4$-coloring and it is a desired hypergraph.
}\end{proof}
%
%

\subsection{Answering to a stronger version of Problem~\ref{prob:4}}
Our aim in this subsection is to present a $3$-uniform $3$-colorable hypergraph  having a complete $6$-coloring
but  no complete $t$-coloring for each $t\in \{4,5\}$. To find such a hypergraph, 
we first made a complete $3$-uniform hypergraph $H$ of order $6$ with size $\binom{6}{3}$ 
so that for any triad of vertices, there is a hyperedge including all of them. 
Next, we tried to generate new hypergraphs by splitting every vertex into two vertices and examine the other necessary properties using a special computer search. By this way, we succeeded to prove the following assertion. 
This method was already used to  make the hypergraph stated in the proof of Theorem~\ref{thm:9v}. 
\begin{thm}\label{thm:stronger-PB}
{There exists a  $3$-uniform  hypergraph $H$ with a complete $\chi(H)$-coloring and a complete $\psi(H)$-coloring, 
but no  complete $t$-coloring for every $t$ satisfying $\chi(H)< t<\psi(H)$ in which  $\psi(H)\ge 2\chi(H)$.
}\end{thm}
\begin{proof}
{Let $H$ be the $3$-uniform hypergraph whose incidence graph is shown in Figure~\ref{fig:12-46-comp3}.
If $H$ has a complete $k$-coloring for $k \ge 7$,
then it has at least thirty-five hyperedges.
However, $H$ has exactly twenty hyperedges and hence $\psi(H) \le 6$.
In fact, $H$ has a complete $3$-coloring and a complete $6$-coloring (see Figures~\ref{fig:12-46-comp3} and~\ref{fig:12-46-comp6}, respectively). Therefore, $\chi(H) = 3$ and $\psi(H)=6$.

\begin{figure}[h]
 \begin{minipage}{0.5\hsize}
  \centering
   \includegraphics[scale = 0.5]{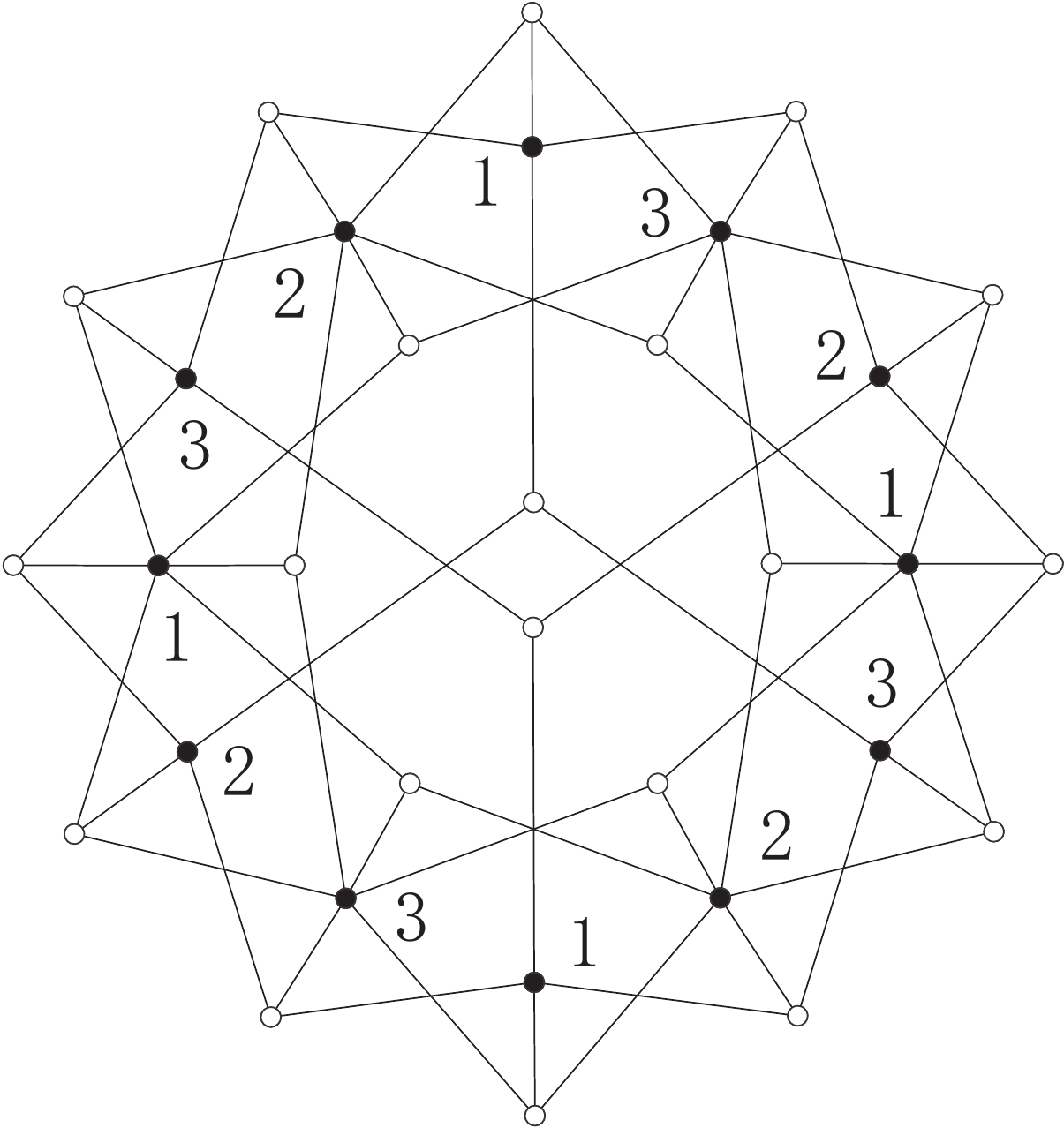}
  \caption{A complete $3$-coloring of $H$.}
  \label{fig:12-46-comp3}
 \end{minipage}
 \begin{minipage}{0.5\hsize}
 \centering
  \includegraphics[scale = 0.5]{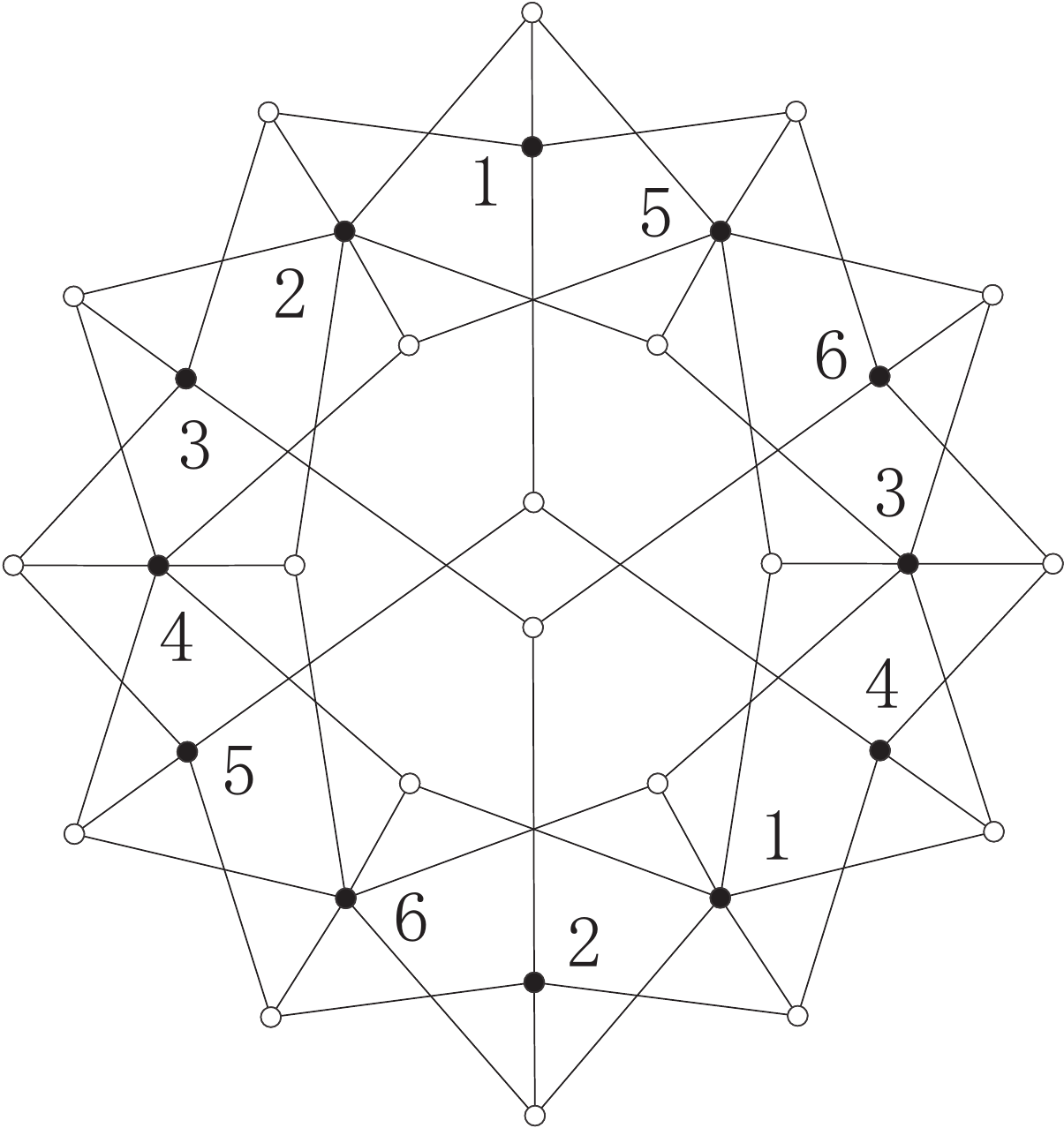}
  \caption{A complete $6$-coloring of $H$.}
  \label{fig:12-46-comp6}
 \end{minipage}
\end{figure}

Next, we show that $H$ has neither a complete $4$-coloring nor a complete $5$-coloring.
According to the construction of $H$, it is not hard to check that there are exactly three independent sets  $X_1, X_2$ and $X_3$ of size four (which shown as vertices numbered by $1, 2$ and $3$ in Figure~\ref{fig:12-46-comp3}, respectively).
Moreover, every independent set  of size three must be a subset of $X_1$, $X_2$, or $X_3$.
Suppose, to the contrary, that $H$ has a complete $4$-coloring using colors $c_1,\ldots, c_4$.
 First, we assume that there exists a color appearing on at least four vertices of $H$, say color $c_1$.
Since $H$ has no independent sets of size  five, the color $c_1$ must appear on all four vertices of a set $X_i$, where $i\in \{1,2,3\}$. Since these four vertices cover all hyperedges of $H$,
the triad $\{c_2, c_3, c_4\}$ does not appear on any hyperedge of $H$, a contradiction.
Now, since $H$ has $12$ vertices, we may assume that every color appears on exactly three vertices of $H$.
On the other hand, $H$  has at most three disjoint  independent sets of size three, a contradiction.
Therefore, $H$ has no complete $4$-coloring.

Suppose, to the contrary, that $H$ has a complete $5$-coloring using colors $c_1,\ldots, c_5$.
As we have observed above, no color can appear on at least four vertices.
Since $H$ has $12$ vertices, there must be a color appearing on exactly three vertices of $H$, say $c_1$.
Call the set of all vertices having the color $c_1$ by $S$.
Since the size of $S$ is three, it must be a subset of $X_1$, $X_2$, or $X_3$, say $X_1$.
 We may assume that  the unique vertex in $X_1\setminus S$ is colored by $c_2$.
Since $X_1$ covers all hyperedges of $H$, the triad $\{c_3,c_4,c_5\}$ does not appear  on any hyperedge of $H$, a contradiction.
Therefore, $H$ has no complete $5$-coloring and it is a desired hypergraph.
}\end{proof}
%
%
%
\section{An exceptional example for  Conjecture~\ref{conj}}
A counterexample of Conjecture~\ref{conj} is given in the following theorem which 
 answers Problem~\ref{prob:3} as well.
This hypergraph was first found
 by writing a C\texttt{++} code for checking complete coloring of hypergraphs and 
by applying it on the specified outputs of {\it plantri} program due to Brinkmann and McKay~\cite{Brinkmann-McKay}.
Note that this face hypergraph is unique by searching among all $3$-colorable planar triangulations on up to $23$ vertices. 
\begin{thm}\label{thm:triangulation}
{There is a $3$-uniform  $3$-colorable face hypergraph of order $12$, obtained from a planar triangulation, having a complete $6$-coloring but with no complete $5$-coloring.
}\end{thm}
\begin{proof}
{Let $H$ be the $3$-uniform face hypergraph obtained from the planar triangulation shown in Figure~\ref{fig:12-facehyp-comp3}. 
If $H$ has a complete $k$-coloring for $k \ge 7$, then $H$ has at least thirty-five hyperedges.
However, $H$ has exactly twenty hyperedges and hence $\psi(H) \le 6$.
In fact, $H$ has a complete $3$-coloring and 
a complete $6$-coloring (see Figures~\ref{fig:12-facehyp-comp3} and~\ref{fig:12-facehyp-comp6}, respectively).
 Therefore, $\chi(H) = 3$ and $\psi(H)=6$.

\begin{figure}[h]
 \begin{minipage}{0.5\hsize}
  \centering
   \includegraphics[scale = 0.5]{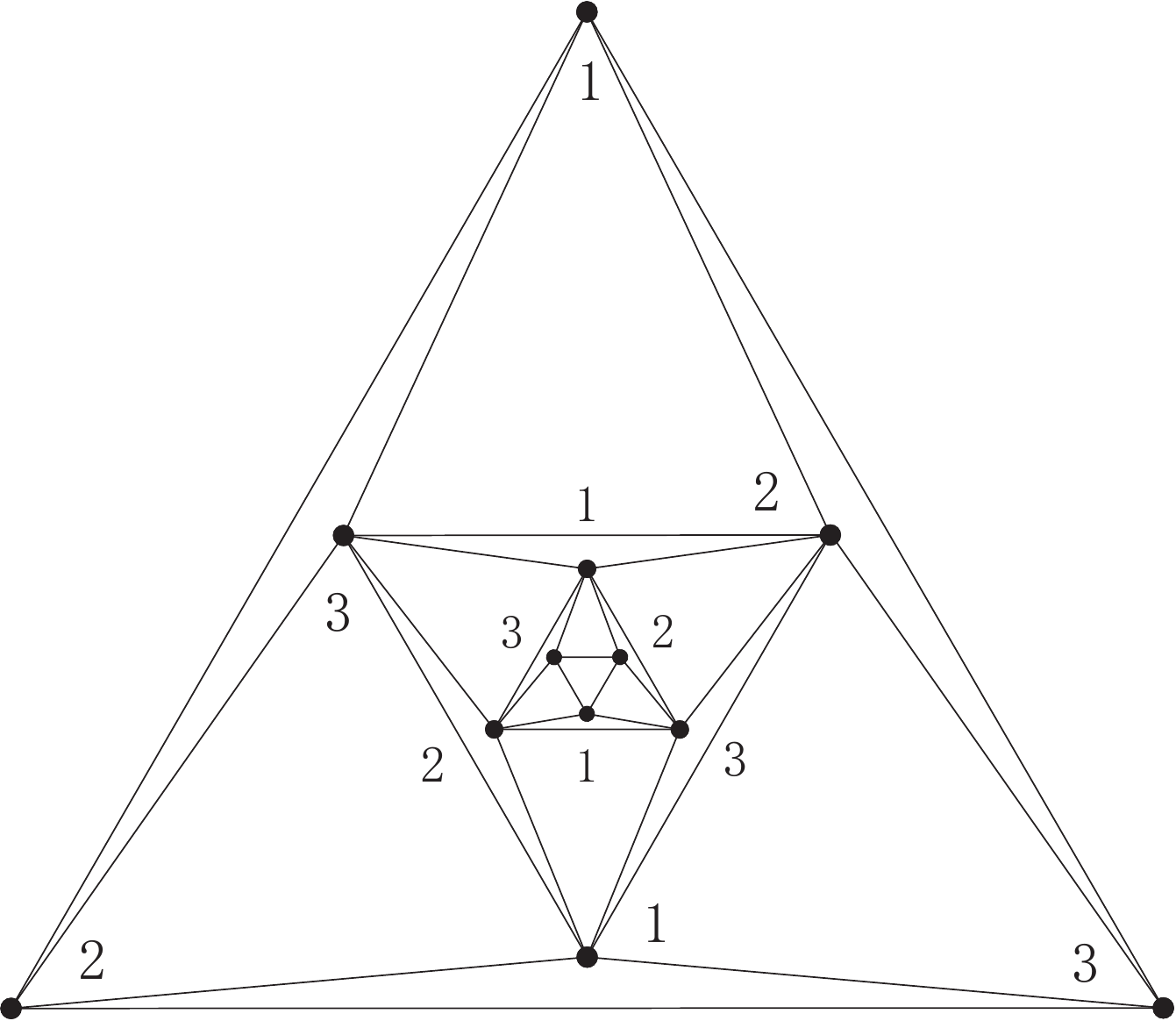}
  \caption{A complete $3$-coloring of $H$}
  \label{fig:12-facehyp-comp3}
 \end{minipage}
 \begin{minipage}{0.5\hsize}
 \centering
  \includegraphics[scale = 0.5]{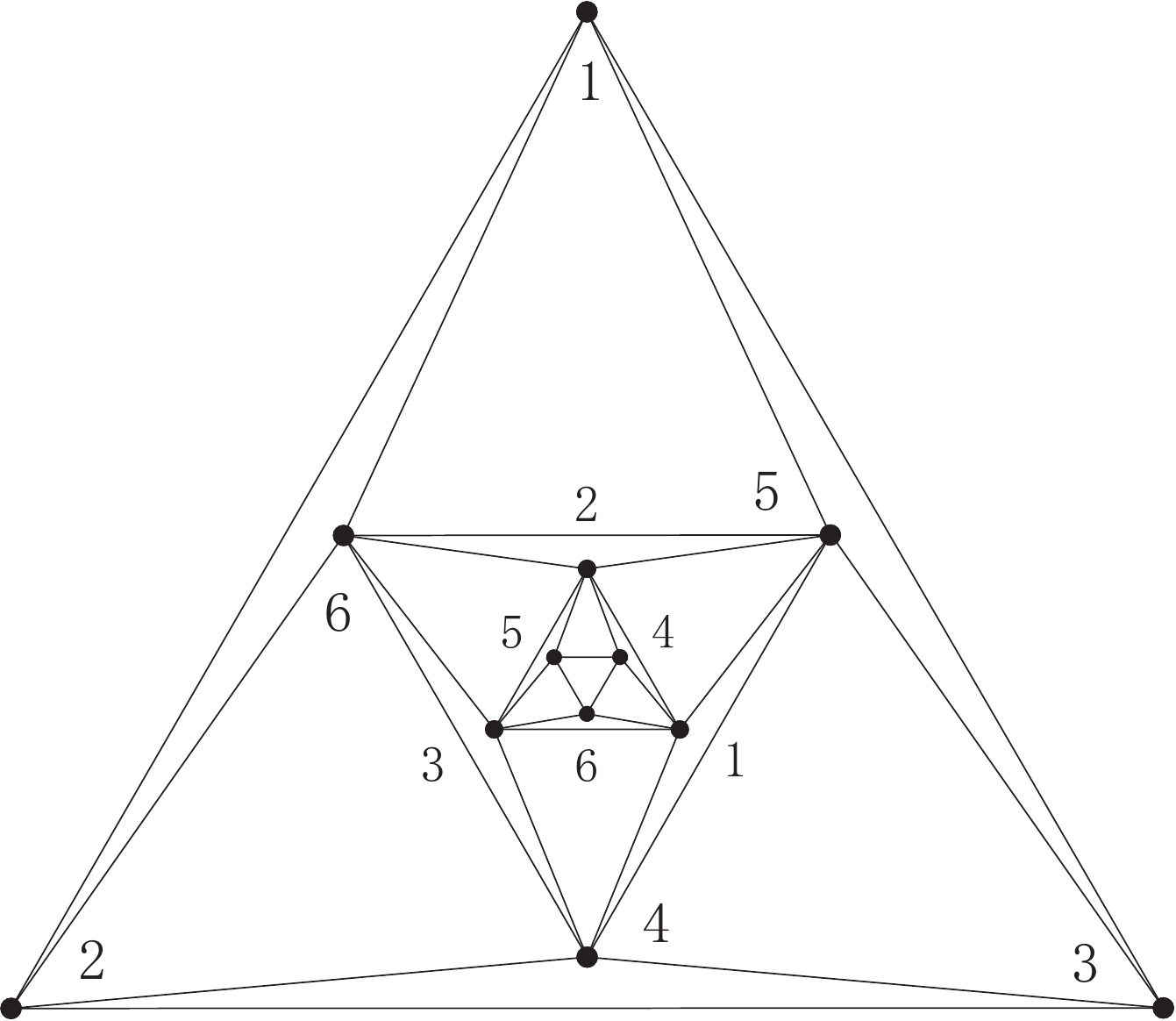}
  \caption{A complete $6$-coloring of $H$}
  \label{fig:12-facehyp-comp6}
 \end{minipage}
\end{figure}

Suppose, to the contrary, that $H$ has a complete $5$-coloring using colors $c_1,\ldots, c_5$.
For every $i$ with $1\le i\le 6$, we call those two vertices of $H$ specifying by the number $i$ in Figure~\ref{fig:12-facehyp-comp6} 
by $v_i$ and $w_i$ such that  $w_i$ is the inner one.
We  may assume that $w_1$, $w_2$, and $w_3$ are colored by $c_1$, $c_2$, and $c_3$, respectively.
We may also assume that each of the colors  $c_4$ and $c_5$ appears on at least one of $w_4$, $w_5$, and $w_6$; 
 otherwise, it is enough to change the colors of them to make this property along with maintaining the property of complete $5$-coloring.
According to the features of the hypergraph $H$, we can also assume that $w_4$, $w_5$, and $w_6$ 
are colored by $c_4$, $c_5$, and $c_2$, respectively.
It is not difficult to check that for a given arbitrary proper coloring of the octahedron, 
every pair of colors is contained in at most two kinds of triads appeared on faces of the octahedron.
Thus the octahedron $v_1v_2\cdots v_6$ has at most two kinds of colored faces including  both of $c_3$ and $c_4$.
Since there exist three remaining triads containing $c_3$ and $c_4$, one can 
conclude that the color $c_4$ must appear on  either $v_4$  or $v_6$.
Similarly, with respect to the colors $c_1$ and $c_5$ on this octahedron, one can also conclude that 
the color $c_5$ must appear on either $v_4$  or $v_5$.
To complete the proof, we shall consider three cases.
\vspace{2mm}
\\
{\bf Case A:} The vertex $v_4$ is colored by $c_2$.
\\
In this case, the vertices  $v_5$ and $v_6$ must be colored by $c_5$ and $c_4$, respectively.
Since  at least one face is colored by the triad $\{c_1,c_4,c_5\}$, the color $c_1$ must also appear on the vertex $v_1$.
Consequently, it is easy to see that  the triad $\{c_3, c_4, c_5\}$ cannot appear,  which is a contradiction.
\vspace{2mm}
\\
{\bf Case B:} The vertex $v_4$ is colored by $c_4$.
\\
In this case, the vertex $v_5$ must be colored by $c_5$ and so the vertex $v_6$ must be colored by $c_1$.
Since at least one face is colored by the triad $\{c_3,c_4, c_5\}$, the color $c_3$ must also appear on the vertex $v_3$.
Consequently, it is easy to see that  the triad $\{c_1,c_3,c_5\}$ cannot appear  which is again a contradiction.
\vspace{2mm}
\\
{\bf Case C:} The vertex $v_4$ is colored by $c_5$.
\\
The proof of this case is similar to  Case B (by exchanging the colors $c_4$ and $c_5$ and using the symmetry of $H$).

Hence the  proof is completed.
}\end{proof}
%
%
%
%
%
%
%
%
%
%

\end{document}